\documentclass[12pt, twoside]{article}
\usepackage{amsmath,amsthm,amssymb}
\usepackage{times}
\usepackage{enumerate}

\def\subjclass#1{{\renewcommand{\thefootnote}{}%
\footnote{\hspace{-0.6cm}\emph{Mathematics Subject Classification
(2010):} #1}}}
\def\subj#1{{\renewcommand{\thefootnote}{}%
\footnote{\hspace{-0.6cm}\emph{keywords:} #1}}}
\def\sqw{\hbox{\rlap{\leavevmode\raise.3ex\hbox{$\sqcap$}}$%
\sqcup$}}
\def\sqb{\hbox{\hskip5pt\vrule width4pt height6pt depth1.5pt%
\hskip1pt}}

\newtheorem{theorem}{Theorem}[section]
\newtheorem{corollary}[theorem]{Corollary}
\newtheorem{proposition}[theorem]{Proposition}
\newtheorem{lemma}[theorem]{Lemma}
\numberwithin{equation}{section}

\numberwithin{equation}{section}

\begin{document}


\baselineskip=17pt


\title{\bf\textsc{ON FACTORIZATION OF $p$-ADIC MEROMORPHIC FUNCTIONS}}

\author{B. SAOUDI, A. BOUTABAA and T. ZERZAIHI }

\date{February 6, 2019}

\maketitle

 \subj{p-adic meromorphic function, primeness , right-primeness, left-primeness}

\begin{abstract}
In this paper, we study  primeness and pseudo primeness of $p$-adic meromorphic functions. We also consider left (resp. right ) primeness of these functions. We give, in particular, sufficient conditions for a meromorphic function to satisfy such properties. Finally, we consider the problem of permutability of entire functions.    
\end{abstract}

\section{Introduction }

For every prime number $p$, we denote by $\mathbb{Q}_{p}$ the field of $p$-adic numbers, and we denote by $\mathbb{C}_{p}$ the completion of an algebraic closure of $\mathbb{Q}_{p}$, which is endowed with the usual $p$-adic absolute value. \ Given $a\in\mathbb{C}_{p}$ and $r>0$, \ $d(a,r)$ and $d(a,r^-)$ are respectively the disks $\{x\in\mathbb{C}_{p} \ / \ |x-a|\leq r\}$ and $\{x\in\mathbb{C}_{p} \ / \ |x-a|< r\}$; and  $C(a,r)$ is the circle  $\{x\in\mathbb{C}_{p} \ / \ |x-a|= r\}=d(a,r)\setminus d(a,r^-)$. It is easily seen that,  two disks have a non-empty intersection if and only if they are nested (i.e, one of them is contained in the other).


We denote by $\mathcal{A}(\mathbb{C}_{p})$ the $\mathbb{C}_{p}$-algebra of entire functions in $\mathbb{C}_{p}$ and $\mathcal{M}(\mathbb{C}_{p})$ the field of meromorphic functions in $\mathbb{C}_{p}$, i.e, the field of fractions of $\mathcal{A}(\mathbb{C}_{p})$. \ Let $f(x)=\sum_{n\geq 0}a_{n}x^{n} $ be a $p$-adic entire function. For all $r>0$, we denote by $|f|(r)=\max_{n\geq 0} |a_{n}|r^{n}$, the maximum modulus of $f$. This is extended to meromorphic functions $h=f/g$ by $|h|(r)=|f|(r)/|g|(r)$. 
\medskip

An element $f\in \mathcal{A}(\mathbb{C}_{p})$ (resp. $f\in \mathcal{M}(\mathbb{C}_{p}) $)  is said to be \emph{transcendental} if it is not a polynomial (resp. \ rational function). \ Thus, a $p$-adic transcendental meromorphic function admits  infinitely many zeros or  infinitely many poles or both. It should also be noted  that a transcendental entire function $f$ has no \emph{exceptional value}  or  \emph{Picard value} in $\mathbb{C}_{p}$,\ so for every 
 $\beta \in \mathbb{C}_{p}$, \ the function $f-\beta$ has infinitely many zeros.
\medskip

Recall also that, if $D$ is a disk and $f$ an analytic function in $D$, then $f(D)$ is  a disk. Moreover if $f^{'}$ has no zero in $D$, there exist a disk $d\subset D$ such that the restiction $f_{\mid d}:d\longrightarrow f(d)$  is bi-analytic. This means that $f_{\mid d}$ is an analytic bijection and that its reciprocal function $f_{\mid d}^{-1}$ is also analytic $[4]$.  \medskip

Let $F, f, g$ be $p$-adic meromorphic functions such that $F=f\circ g$. \ We say that $f$ and $g$ are respectively  left and right factors of $F$. The function $F$ is said to be prime (resp. pseudo-prime, resp. left-prime, resp. right-prime)  if every factorization of $F$ of the above form implies that either $f$  or $g$ is a linear rational function (resp. $f$ or $g$ is  a rational function, resp. $f$ is  a linear rational function  whenever $g$ is transcendental, $g$  is linear whenever $f$ is transcendental). If the factors are restricted to entire functions, the factorization is said to be in the entire sense. \ The article by B\'ezivin and Boutabaa [2] is to our knowledge the only work dedicated to the study of factorization of  $p$-adic meromorphic functions. They show, in particular, that if an entire function $F$ is prime (resp pseudo-prime) in $\mathcal{A}(\mathbb{C}_{p})$, then $F$ is prime (resp pseudo-prime) in $\mathcal{M}(\mathbb{C}_{p})$. This is false in the field $\mathbb{C}$ of complex numbers. \ Indeed, Ozawa in [6] gives examples of complex entire functions that are prime in $\mathcal{A}(\mathbb{\mathbb{C}})$ without being prime in $\mathcal{M}(\mathbb{\mathbb{C}})$.
 
 In this article, we provide sufficient conditions for a $p$-adic meromorphic function to be prime or pseudo-prime.  We give examples of meromorphic functions satisfying these conditions.
 
  We also show that almost all $p$-adic transcendental entire functions are prime, in the sense that it is most often enough to add or multiply these functions by an affinity to obtain a prime entire  function.  

Finally, we briefly discuss the question of permutability of entire functions. Or, in other words, one wonders: when do we have $ f \circ g = g \circ f $ for $p$-adic entire functions $f$ and $g$? \ 

Our method is based on the distribution of zeros and poles of the considered functions  as well as the properties of their maximum modulus.

\medskip
      
\section{PRIMALITY AND  PSEUDO PRIMALITY}
Let us first prove the following result which gives sufficient conditions for a $p$-adic meromorphic function to be pseudo-prime.

\begin{theorem} 
	 Let $F$ be a  transcendental  meromorphic function in $\mathbb{C}_{p}$ whose all poles are simple except a finite number of them. Suppose moreover that, for every $\beta\in\mathbb{C}_{p} $, all the zeros of the function $F-\beta$ are simple except a finite number of them. Then $F$ is pseudo-prime. 
\end{theorem}
To prove Theorem $2.1$ we need the following lemma, whose proof is given in $[ 5 ]$
\begin{lemma}
Let $f_{1}, f_{2}\in\mathcal{A}(\mathbb{C}_{p})$ such that $f_{1}^{'}f_{2}-f_{1}f_{2}^{'}\equiv c ;\,\, \, \, c\in \mathbb{C}_{p}$. So, if one of the functions $ f_{1}, f_{2}$ is not affine then $c=0$ \ and \ $f_{1}/f_{2}$ is a constant.  
\end{lemma}
We also need the following lemma, whose proof is given in $[ 2 ]$
\begin{lemma}
	Let $F,f,g$ be three meromorphic functions in $\mathbb{C}_{p}$. Suppose that $f$ is not a rational function and that $F=f\circ g $. Then g entire.
\end{lemma}
\noindent{\it Proof of Theorem} \text{2.1. }\\
	Suppose that $F$ is not pseudo-prime. Hence there exist two  transcendental  meromorphic functions $f$ and $g$ such that $ F= f \circ g$. Then, by  Lemma $2.3$, the function  $g$ is entire. Let us write $f$ in the form $f=f_{1}/f_{2}$ where  $f_{1}$,  $f_{2}$ are entire functions with no common zeros. As $f$ is transcendental, we see that at least one of the functions $f_{1}$, $f_{2}$ is  transcendental. So by Lemma $2.2$, \ $f_{1}^{'}f_{2}- f_{1}f_{2}^{'}$   is a non-constant entire function, hence admits at least a zero $\alpha$. Let us distinguish the two following cases:  
	
	\medskip  
	\textbf{1}. If  $f_{2}(\alpha)=0$, then $f_{1}(\alpha)\not=0$ and $f_{2}^{'}(\alpha)=0$, so $\alpha$ is a multiple zero of $f_{2}$ and all element of the set $g^{-1}(\{\alpha\})$ are multiple zeros of $f_{2}\circ g$ and are multiple poles of $F$. Since $g\in\mathcal{A}(\mathbb{C}_{p})\backslash\mathbb{C}_{p}[X]$, then  the set $g^{-1}(\{\alpha\})$ is infinite and $F$ has infinitely many multiple poles, which is  a contradiction.
	
	\medskip
	\textbf{ 2}. If $f_{2}(\alpha)\not=0$, \ then   $f^{'}(\alpha)=(f_{1}^{'}f_{2}- f_{1}f_{2}^{'})(\alpha)/f_{2}(\alpha)=0$. \ Let $\beta=f(\alpha)$, then $\alpha$ is a multiple zero of $f-\beta$.
	But the equation $g(x)=\alpha$ admits  infinitely many solutions and for every such solution  $\omega$ we have: \\
	
	$\left\{\begin{array}{ll}
(F-\beta)(\omega)=F(\omega)-\beta=f\circ g (\omega) -\beta= f(\alpha)-\beta=0,\\ 
	(F-\beta)^{'}(\omega)=F^{'} ( \omega )=f^{'}(g(\omega))\times g^{'}(\omega)=f^{'}(\alpha)\times g^{'}(\omega)=0.\\
	\end{array}\right.$
	
\smallskip

	\noindent Then $F-\beta$ has a infinitely many zeros, which is a contradiction again.\\
	Hence $F$ is pseudo-prime. 	\qquad\qquad\qquad\qquad\qquad\qquad\qquad\qquad\qquad\qquad\qquad\sqw
	
\bigskip

Theorem 2.1. provides, in particular, necessary conditions for the pseudo-primeness of $p$-adic entire functions, which are summarized in the following corollary:
\begin{corollary}
	Let $F\in \mathcal{A}(\mathbb{C}_{p}) \backslash\mathbb{C}_{p}[X]$ be  such that, for every $\beta\in\mathbb{C}_{p}$, \ all the zeros of the function $F-\beta$ are simple except  a finite number of them. \ Then $F$ is pseudo-prime. 
\end{corollary}
The following results  give more information about the left factor in any factorization of a $p$-adic  entire function that satisfies the above conditions.
\begin{theorem}
	Let $F\in \mathcal{A}(\mathbb{C}_{p}) \backslash\mathbb{C}_{p}[X] $ be such that for every $\beta\in\mathbb{C}_{p}$, only finitely many zeros of $F-\beta$ are multiple. Then $F$ is left-prime. 
\end{theorem}
\begin{proof}
	Suppose that $F=f \circ g $ where $g\in \mathcal{A}(\mathbb{C}_{p}) \backslash\mathbb{C}_{p}[X]$ and $f\in \mathcal{A}(\mathbb{C}_{p})$. By Corollary $2.4$, we know that $F$ is pseudo-prime. So  $f$ is a polynomial. Suppose that  $\deg f \geq 2$, then $f^{'}$ has at least one zero $\alpha$. Since $g\in \mathcal{A}(\mathbb{C}_{p}) \backslash\mathbb{C}_{p}[X] $, the set $W=\{x\in\mathbb{C}_{p} \ / \ g(x)=\alpha \}$ is infinite.  \\
	Let $\beta=f(\alpha)$. \ Then, for every $\omega\in W$, we have:

	$\left\{\begin{array}{ll}
	((F-\beta)(\omega)=F(\omega)-\beta=f\circ g (\omega) -\beta= f(\alpha)-\beta=0,\\ 
	(F-\beta)^{'}(\omega)=F^{'} ( \omega )=f^{'}(g(\omega))\times g^{'}(\omega)=f^{'}(\alpha)\times g^{'}(\omega)=0.\\
	\end{array}\right.$
	
	\smallskip
	
	\noindent This means that \ $F-\beta$ has infinitely many multiple zeros, a contradiction.  Hence the left-primeness of $f$  is proven.
\end{proof}
\begin{theorem}
	Let $F\in \mathcal{A}(\mathbb{C}_{p}) \backslash\mathbb{C}_{p}[X]$ be  such that for every $\beta\in\mathbb{C}_{p} $, the function $F-\beta$ has at most one multiple zero. Then $F$ is prime.  
\end{theorem}
\begin{proof}
	By Theorem $2.5$, we already see that $F$ is left-prime. So,  it  remains to  show the right-primeness of $F$. For that,  suppose that $F(z)=f\circ g$, where  $f\in\mathcal{A}(\mathbb{C}_{p}) \backslash\mathbb{C}_{p}[X]$ and $g\in \mathcal{A}(\mathbb{C}_{p})$. From Corollary $2.4$, we know that $F$ is pseudo-prime. So  $g$ is a polynomial. Suppose that  $\deg g=d \geq 2$. We have
	\ $F'(z)=f^{'}(g(z))g^{'}(z).$ \ 
	Since $f\in \mathcal{A}(\mathbb{C}_{p}) \backslash\mathbb{C}_{p}[X] $,  the function  $f^{'}$ \ has infinitely many zeros. So we may choose an element $w\in\mathbb{C}_{p}$ such that $f^{'}(w)=0$ and $g-w$ has only simple zeros $\gamma_{1}, ...,\gamma_{d}$. \ Then for $i=1,...,d$, we have  
	$ \left\{
	\begin{array}{ll}
	F(\gamma_{i})=f(w)=\beta\\ 
	(F-\beta)^{'}(\gamma_{i})=0\\ 
	\end{array} \right.$,      
	\  which means $\gamma_{1}, ...,\gamma_{d}$ are multiple zeros  of $F-\beta$, a contradiction. Hence $F(z)$ is right-prime.
\end{proof}
\begin{theorem}
	Let $F$ be a $p$-adic  transcendental  meromorphic function that admits at most finitely many poles. Suppose that for every $\beta\in\mathbb{C}_{p} $, only finitely many zeros of $F-\beta$ are multiple. Then $F$ is right-prime. 
\end{theorem}
To prove this theorem, we need the following lemma whose proof is given in B\'ezivin $[3]$. It is more general  than Lemma $2.2.$
\begin{lemma}
	Let $n\geq 1$ and let $f_{1}, \dots,  f_{n}$ be $p$-adic entire functions  such that the wronskian $W(f_{1}, \dots,  f_{n} )$ is a non-zero polynomial. Then $f_{1}, \dots,  f_{n}$ are polynomials.  
\end{lemma}

\noindent{\it Proof of Theorem} \text{2.7. }
	   
	Suppose that $F(z)=f\circ g$, where  $f\in\mathcal{M}(\mathbb{C}_{p}) \backslash\mathbb{C}_{p}(X)$ and $g\in\mathcal{M}(\mathbb{C}_{p})$. By Theorem $2.1$ the function $F$ is pseudo-prime. Then $g$  is a polynomial function. Suppose that $ deg\, g  \geq 2$. Let us write $f$ in the form $f={f_{1}}/{f_{2}}$ \ where  $f_{1}$,  $f_{2}$ are entire functions with no common zeros. \ We have \ $\displaystyle f'(z)=W(f_{2},f_{1})/{f_{2}^{2}})(g(z))g^{'}(z).$ \  As $f$ is transcendental and has a finite number of poles , we see that $f_{1}$ is transcendental and $f_{2}$ is polynomial. \ As $f \in\mathcal{M}(\mathbb{C}_{p}) \backslash\mathbb{C}_{p}(X)$, \ it follows by Lemma 2.8 that $W(f_{2},f_{1})\in \mathcal{A}(\mathbb{C}_{p}) \backslash\mathbb{C}_{p}(X) $ \ and  admits then infinitely many zeros $\{w_{n}\}$  ( which are not  zeros of $f_{2}$ . For every integer $n$, big enough,  the equation $g(z)=w_{n}$ admits at least two distinct roots , which are also   common roots of
$\displaystyle \left\{
	\begin{array}{ll}
	F(z)=f(w_{n})\\ 
	F^{'}(z)=0\\
	\end{array} \right.$. \ 
Then we have a contradiction and $F(z)$ is right-prime.  \qquad\qquad\qquad\qquad\qquad\quad	\sqw
	
	\bigskip
	
	\begin{theorem}
		Let $f(x)=\sum_{n\geq N}a_{n}x^{n}$ be a $p$-adic  entire function such that $a_{n} \neq0$, for every $n\geq N$. Suppose that there exists an integer  $n_{0} \geq N $  such that   $(|a_{n}/a_{n+1}|)_{n\geq n_{0}} $ is a strictly increasing unbounded sequence.\ Then the function $f$ is pseudo-prime.     
	\end{theorem}
	
	\medskip
	
	\noindent To prove Theorem 2.9 we need the following lemma whose proof is given in [4].  
	
	\begin{lemma}
		Let $f$ be a $p$-adic entire function defined by  $f(x)=\sum_{n\geq N}a_{n}x^{n}$. \ Let $\mu$ and $\nu$  be, respectively, the smallest integer and the largest one such that $\displaystyle|f|(r)=|a_\mu|r^\mu=|a_\nu|r^\nu$. \ Then:
		
		\ \ i) \ \ $\mu$ \ is the number of zeros of $f$ in the disk $d(0, r^-)$;
			
			\ ii) \ \  $\nu$ \ is the number of zeros of $f$ in the disk $d(0, r)$;
			
			iii) \ \ $\nu-\mu$ \ is the number of zeros of $f$ in the circle ${\mathcal C}(0, r)$.
	
	\end{lemma}
	\noindent{\it Proof of Theorem 2.9.} 

	Let us, first set $r_{n}=|a_{n}/a_{n+1}|, \  \forall n \geq 0$. \ Quit to replace $n_0$ by a greater integer, we may assume that $n_0$ is such that: \ $r_{n_{0}}>\max\{r_{0}, \dots, r_{n_{0}-1}\}$.
\ In the sequel, we will show  that $f$ has only simple zeros in $\mathbb{C}_{p} \backslash d(0,r_{n_{0}})$. \ Indeed, let $r>r_{n_{0}}$. \  We distinguish two cases:

		\noindent{\it \ i}) \ $r=r_n$ for some $n\geq n_{0}+1$. As the sequence $(r_k)_{k\geq n_0}$ is strictly increasing, we have $r_{n-1}<r_{n}<r_{n+1}$. We then have:
			
			 $|a_{n+1}|r_{n}^{n+1}/|a_{n}|r_{n}^{n}=|a_{n+1}/a_{n}|r_{n}=1$, which means that  $|a_{n}|r_{n}^{n}=|a_{n+1}|r_{n}^{n+1}$.\\
			Moreover for every integer $\ell$,  $0\leq \ell \leq n-1$, we have:
			
		$ |a_{\ell}|r_{n}^{\ell}/|a_{n}|r_{n}^{n}=|a_{\ell}/a_{n}|{1}/{r_{n}^{n-\ell}}= 
			|a_{\ell}/a_{\ell+1}|\dots |a_{n-1}/a_{n}|{1}/{r_{n}^{n-k}} <1 $.\\
			Finally, for every integer $\ell$, \   $\ell>n+1$, we have:
			
				$|a_{\ell}|r_{n}^{\ell}/|a_{n+1}|r_{n}^{n+1}= 
			|a_{\ell}/a_{n+1}|r_{n}^{\ell-n-1}$
			
			\qquad\qquad\qquad\qquad$<|a_{\ell}/a_{n+1}||a_{n+1}/a_{n+2}|\dots |a_{\ell-1}a_{\ell}|=1.$
			
		\noindent	Hence $|f|(r_{n})=\max_{\ell\geq 0} |a_{\ell}|r_{n}^{\ell}$ is reached for the two values   $\ell=n$ and $\ell=n+1$. This implies, by Lemm 2.10.,  that  $f$ has only one zero in the circle $ C(0,r_{n}).$\\
	 
		\noindent{\it ii}) \ Suppose now that $r$ is different from $r_n$ for all $n>n_0$.  Let  $n\geq n_{0}$ be the sole integer such that $r_{n}<r<r_{n+1}$. We then have, for every integer $\ell$,  $0\leq \ell \leq n $:\\  $\displaystyle|a_{\ell}|r^\ell=|\dfrac{a_{\ell}}{a_{\ell+1}}||\dfrac{a_{\ell+1}}{a_{\ell+2}}| \dots|\dfrac{a_{n}}{a_{n+1}}||a_{n+1}|r^{\ell}
			\leq|a_{n+1}|r_{n}^{n+1-\ell} r^{\ell}<|a_{n+1}|r^{n+1}$.
		
		\noindent	Moreover for every integer
		 $\ell>n+1$, we have:
		\\   
			$\displaystyle \dfrac{|a_{\ell}|r^{\ell}}{|a_{n+1}|r^{n+1}} = \dfrac{|a_{\ell}|}{|a_{n+1}|}r^{\ell-n-1}
			< \dfrac{|a_{\ell}|}{|a_{n+1}|}r_{n+1}^{\ell-n-1}\leq |\dfrac{a_{\ell}}{a_{n+1}}||\dfrac{a_{n+1}}{a_{n+2}}|\dots
			|\dfrac{a_{\ell-1}}{a_{\ell}}|=1.$
			
			\noindent Then $|f|(r)=\max_{\ell\geq 0} |a_{\ell}|r^{\ell}$ is reached 
			only for $\ell=n+1$. This implies, by Lemma 2.10.,  that $f$ has no zero in the circle $ C(0,r)$. \  It follow that all the zeros of  $f$ in $\mathbb{C}_{p} \backslash d(0,r_{n_{0}})$ are simple.  \
		Now,  for every $\beta\in \mathbb{C}_{p},$ there exists $r_{\beta}>0$ such that: \ {  $\displaystyle|f-\beta|(r)=|f|(r), \,\, \forall  r>\max (r_{n_{0}},r_{\beta}).$ } \ It follows that  $f-\beta$ has only simple zeros  in $\mathbb{C}_{p} \backslash d(0, \max (r_{n_{0}},r_{\beta}))$.  
		\ This means that all the possible multiple zeros of  $f-\beta $ lie in $d(0, \max (r_{n_{0}},r_{\beta}))$ and are therefore  finitely many. 
	\ Using Corollary 2.4, \ we complete the proof of Theorem 2.9. \qquad\qquad\quad \sqw

	\begin{corollary}
	Let $f$ be a $p$-adic entire function satisfying the conditions of  Theorem $2.9.$, then for any non-zero polynomial $P$, the $p$-adic meromorphic  function $ g=f/P$ is  pseudo-prime.
\end{corollary}
\begin{proof}
	Note first that we may suppose that $f$ and $P$ have no common zeros. It is clear that $g$ has finitely many poles.  Now let   $\beta \in \mathbb{C}_{p}$. We see that the zeros of $g-\beta$ are the same as those of$f-\beta P$. Moreover, there exists $r_{\beta}>0$ such that for every $r>r_{\beta}$ we have: 
	$\displaystyle |f-\beta P|(r)= |f|(r).$ It follows, by Theorem 2.9,  that the function $f-\beta P$ (and therefore also $g-\beta$) has at most a finite number of multiple zeros.
	Thus by Theorem $2.1.$  the meromorphic function $g$ is pseudo-prime.

\end{proof} 

\begin{corollary}
Let $f(x)=\sum_{n\geq N}a_{n}x^{n}$ and $g(x)=\sum_{n\geq N}b_{n}x^{n} $ be two $p$-adic entire functions such that $a_{n} b_{n}\neq 0 , \forall n\geq N $. Let  $n_{0} \geq N $ be an integer such that the sequences  $(|a_{n}/a_{n+1}|)_{n\geq n_{0}} $ and $(|b_{n}/b_{n+1}|)_{n\geq n_{0}} $ are strictly increasing and unbounded.\ Suppose moreover that $\lim_{r\to+\infty} (|f|(r)/|g|(r))=+\infty$. \ Then the meromorphic function \ $h=f/g$ \ is pseudo-prime.     
\end{corollary}
\begin{proof} Note first that we may suppose that the entire functions $f$ and $g$ have no common zeros. By Theorem 2.9, we see that the entire functions $f$ and  $g$ have at most finitely many multiple zeros. Hence the meromorphic function $h=f/g$ has at most finitely many multiple zeros and poles.
	
	 Now let   $\beta \in \mathbb{C}_{p}$. We see that the zeros of $h-\beta$ are the same as those of$f-\beta g$. Moreover, as $\lim_{r\to+\infty} (|f|(r)/|g|(r))=+\infty$,  there exists $r_{\beta}>0$ such that for every $r>r_{\beta}$ we have: 
	 $\displaystyle |f-\beta g|(r)= |f|(r).$ It follows, by Theorem 2.9,  that the function $f-\beta g$ (and therefore also $h-\beta$) has at most a finite number of multiple zeros.
	 Thus, by Theorem $2.1$,  the meromorphic function $h$ is pseudo-prime.

\end{proof} 	 
In what follows, we will provide examples of meromorphic p-adic functions satisfying the conditions of Corollary 2.12. \ Let's first recall that, given a real number $x$, we call integer part of $x$ and we denote by $E(x)$ the unique integer such that $E(x)\leq x <E(x)+1$. \ \ It is easily shown that:

\medskip

\begin{lemma} 
	
For all real numbers $x$ and $y$, we have: 

\centerline{$ E(x-y)\leq E(x)-E(y)\leq E(x-y)+1.$}
\end{lemma}

\begin{proposition} 
	Let $N$ be an integer$\geq 3$ and let $\alpha, \beta\in\mathbb{C}_{p}$ be such that $|\beta|<|\alpha|<1$. Let $f$ and $g$ the functions defined by  $ f(x)=\sum_{n\geq0}\alpha^{E((n/N)^N)}x^n$ and 
	$g(x)=\sum_{n\geq0}\beta^{E((n/N)^N)}x^n$. The meromorphic function $h={f}/{g}$ is a pseudo-prime.
\end{proposition} 
\medskip
\begin{proof}
	We have \ $f(x)=\sum_{n\geq0}a_nx^n$, where \ $\displaystyle a_n=\alpha^{E((n/N)^N)}$ for every $n\geq0$. 
	
	\noindent We easily check that, $\lim_{n\to+\infty}|a_n|r^n= \lim_{n\to+\infty}|\alpha^{E((n/N)^N)}|r^n=0$,  for every $r> 0$;  which means that $f$ is an entire function in $\mathbb{C}_{p}$. 
	
	\noindent Let us now show that, if $n_0$ is an integer $\geq({2N^{N-1}}/({N-1}))^{{1}/({N-2})}$, \ the sequence $(|{a_n}/{a_{n+1}}|)_{n\geq n_0}$ is strictly increasing.
	
		\noindent For every $n\geq0$, we have: \ \  $|{a_n}/{a_{n+1}}|=({1}/{|\alpha|})^{E(({n+1}/{N})^N)-E(({n}/{N})^N)}$.
		
		\medskip
		
		\noindent As the real function $x\mapsto({1}/{|\alpha|})^x$ \ is strictly increasing, it follows by Lemma 2.10. that:
	\begin{equation}
\biggl(\dfrac{1}{|\alpha|}\biggr)^{E((\frac{n+1}{N})^N-(\frac{n}{N})^N)}\leq
\biggl|\dfrac{a_n}{a_{n+1}}\biggr|\leq\biggl(\dfrac{1}{|\alpha|}\biggr)^{E((\frac{n+1}{N})^N-(\frac{n}{N})^N)+1}
	\end{equation}
		
		\noindent In the same way, we have:	
		\begin{equation}	\biggl(\dfrac{1}{|\alpha|}\biggr)^{E((\frac{n+2}{N})^N-(\frac{n+1}{N})^N)}\leq\biggl|\dfrac{a_n+1}{a_{n+2}}\biggr|\leq\biggl(\dfrac{1}{|\alpha|}\biggr)^{E((\frac{n+2}{N})^N-(\frac{n+1}{N})^N)+1}
		\end{equation}
		
	\noindent It follows from these last two inequalities that:	
	\begin{equation}	
\biggl|\dfrac{a_{n+1}}{a_{n+2}}\biggr|- \biggl|\dfrac{a_n}{a_{n+1}}\biggr|\geq \biggl(\dfrac{1}{|\alpha|}\biggr)^{E((\frac{n+2}{N})^N-(\frac{n+1}{N})^N)}-\biggl(\dfrac{1}{|\alpha|}\biggr)^{E((\frac{n+1}{N})^N-(\frac{n}{N})^N)+1}
	\end{equation}
	
\medskip

But, by Lemma 2.10., We have: 
	
\noindent$\displaystyle E((\frac{n+2}{N})^N-(\frac{n+1}{N})^N)-E((\frac{n+1}{N})^N-(\frac{n}{N})^N)-1\geq$ 

\noindent$\displaystyle \qquad\qquad\qquad\qquad\qquad E([(\frac{n+2}{N})^N-(\frac{n+1}{N})^N]-[(\frac{n+1}{N})^N-(\frac{n}{N})^N])-1$.

	As :\ \ $\displaystyle (\frac{n+2}{N})^N-(\frac{n+1}{N})^N   = \dfrac{1}{N^N}\sum_{i=0}^{N-1}(n+2)^i(n+1)^{N-i-1}$ and

	$\displaystyle (\frac{n+1}{N})^N-(\frac{n}{N})^N=
	\dfrac{1}{N^N}\sum_{i=0}^{N-1}n^i(n+1)^{N-i-1}$, \ we see that:

	\noindent$ [(\frac{n+2}{N})^N-(\frac{n+1}{N})^N]-[(\frac{n+1}{N})^N-(\frac{n}{N})^N]=\dfrac{1}{N^N}\sum_{i=0}^{N-1}(n+1)^{N-i-1}[(n+2)^i-n^i]$. 	
	\noindent Hence:
	
	\noindent$[(\frac{n+2}{N})^N-(\frac{n+1}{N})^N]-[(\frac{n+1}{N})^N-(\frac{n}{N})^N]=$
	
	\qquad	\qquad	\qquad	\qquad$\dfrac{2}{N^N}\sum_{i=0}^{N-1}(n+1)^{N-i-1}\sum_{j=0}^{i-1}(n+2)^jn^{i-j-1}$.
	
It follows that:

\quad$[(\frac{n+2}{N})^N-(\frac{n+1}{N})^N]-[(\frac{n+1}{N})^N-(\frac{n}{N})^N]\geq$

\quad\qquad	\qquad	\qquad\qquad\qquad$\dfrac{2}{N^N}\sum_{i=0}^{N-1}n^{N-i-1}\sum_{j=0}^{i-1}n^jn^{i-j-1}$,

then that:

\qquad\qquad$[(\frac{n+2}{N})^N-(\frac{n+1}{N})^N]-[(\frac{n+1}{N})^N-(\frac{n}{N})^N]\geq\dfrac{2n^{N-2}}{N^N}\sum_{i=0}^{N-1}i$,

and finally:

\qquad\qquad$[(\frac{n+2}{N})^N-(\frac{n+1}{N})^N]-[(\frac{n+1}{N})^N-(\frac{n}{N})^N]\geq\dfrac{n^{N-2}(N-1)}{N^{N-1}}$.

It follows that for every $n\geq n_0$, we have:

\qquad\qquad$[(\frac{n+2}{N})^N-(\frac{n+1}{N})^N]-[(\frac{n+1}{N})^N-(\frac{n}{N})^N]\geq2$.

Consequently for every $n\geq n_0$, we have:

 \qquad\qquad$E((\frac{n+2}{N})^N-(\frac{n+1}{N})^N)-E((\frac{n+1}{N})^N-(\frac{n}{N})^N)-1>0$.
 
 It follows that  for every $n\geq n_0$, we have:
 $\biggl|\dfrac{a_{n+1}}{a_{n+2}}\biggr|> \biggl|\dfrac{a_n}{a_{n+1}}\biggr|$. 
 
 \medskip
 
 We complete the proof by applying Theorem 2.9.
\end{proof} 
\bigskip
In what follows, one aims to show that, given a $p$-adic transcendental entire function  $f$, it is always easy to transform it into a prime entire  function. For this, most often, we have just to add, or multiply  $f$ by an affinity.


\begin{theorem}\quad
	Let $f$ be a $p$-adic  transcendental  entire function. Then the set $\displaystyle \{a\in \mathbb{C}_{p} ;\, f(x) -ax \,\,\,is \, not \,prime\} $
	is at most a countable set.
\end{theorem}
To prove Theorem 2.15 we need the following lemma.  
\begin{lemma}
	Let $ f\in \mathcal{A}(\mathbb{C}_{p})\backslash\mathbb{C}_{p}[X] $. Then there exists a countable set $E\subset \mathbb{C}_{p} $ such that for every $a\in \mathbb{C}_{p}\backslash E $ and every $b\in\mathbb{C}_{p}$, the function $f(x)-(ax+b)$ has at most one multiple zero.
\end{lemma}
\begin{proof} 
	Let $Z(f^{''})$ be the set of zeros of  $f^{''}$. 
Since $\mathbb{C}_{p}$ is a separable space, there exists a countable family of disks $(d_i)_{i\ge1}$ such that $\mathbb{C}_{p}\setminus Z(f^{''})=\cup_{i\ge1}d_i$ and, for every $i\ge1$, the restriction $f'_i$ of $f'$ to $d_i$ is a bi-analytic function on $d_i$.  Then we have 
\ $\displaystyle\mathbb{C}_{p}=(\cup_{i\ge1} D_i) \cup f^{'}(Z(f^{''})).$

\bigskip

Let $g$ be the function defined on $\mathbb{C}_{p}$ by $g(x)=f(x)-xf'(x)$.
It should be noted that even if the family $(d_i)_{i\ge1}$ is chosen so that the disks $d_i$ are pairwise disjoint, there is no guarantee that the family $(D_i)_{i\ge1}$ retains this property.  In other words some of the disks  $D_i$ could be nested. To take account of this fact, let us set: \centerline{$\Gamma=\{(i,j)\in (\mathbb{N}^{\ast})^2 \ / \ D_{i}\subsetneqq D_{j} \ \text{and} \ g\circ(f'_i)^{-1}\neq  g\circ(f'_j)^{-1} \ \text{in} \  D_{i}\}$;}\\ 
$\Delta_{ij}=\{x\in D_i \ / \  g\circ(f'_i)^{-1}(x)=g\circ(f'_j)^{-1}(x) \}, \forall (i,j)\in\Gamma$ \
and $\Delta=\cup_{(i,j)\in\Gamma}\Delta_{ij}$.\\
Then $E=\mathbb{C}_{p}\setminus [(\cup_{i\ge1}D_i)\setminus[(f'(Z(f'')))\cup\Delta]]$ is a countable subset of $\mathbb{C}_{p}$. \ Indeed it is easily seen that $E\subset (f'(Z(f'')))\cup\Delta$.\\
 
Let $a\in\mathbb{C}_{p} \backslash E$. Let $I_a=\{i\in\mathbb{N}^{\ast} \ / \ a\in D_i \}$. \ Hence the disks $D_i$, for $i\in I_a\ $, are nested. We easily show that the set of multiple zeros of the function \ $f(x)-(ax+b)$ \ is equal to $A=\{(f'_i)^{-1}(a) \ / i\in I_a \ \text{and} \ f\circ(f'_i)^{-1}(a)=a(f'_i)^{-1}(a)+b\}$.\\
Suppose that $(f'_i)^{-1}(a)$ and $(f'_j)^{-1}(a)$ are two distinct elements of $A$.  We may assume that $D_{i}\subsetneqq D_{j}$. The fact that each of these elements is a solution of the equation $f(x)-ax=b$ implies that:
$f((f'_i)^{-1}(a)) -(f'_i)^{-1}(a)f'((f'_i)^{-1}(a))=f((f'_j)^{-1}(a)) -(f'_j)^{-1}(a)f'((f'_j)^{-1}(a))$ \ or, \  in other words, \ $g\circ (f'_i)^{-1}(a)=g\circ (f'_j)^{-1}(a)$. \ As $a\notin\Delta$, we deduce that $g\circ (f'_i)^{-1}(x)=g\circ (f'_j)^{-1}(x), \forall x\in D_i$. \\
 By derivation of this last equation, we have:  \ $(f'_i)^{-1}(x)=(f'_j)^{-1}(x),  \forall x\in D_i$ \ and particularly $(f'_i)^{-1}(a)=(f'_j)^{-1}(a)$, which is a contradiction. Hence the set $A$ admits at most one element. This completes the proof Lemma 2.16.
\end{proof}
\noindent{\it Proof of Theorem 2.15 }
	
	Let $E$ be the countable set of Lemme 2.16. and let $a\in \mathbb{C}_{p}\backslash E$. \ Since $f$ is a  transcendental  entire function, we see that the function $h (x) = f (x) -ax$ is also a  transcendental  entire function. Lemma 2.16 then ensures that, for every $b\in\mathbb{C}_{p}$, \  the function $h (x) -b=f (x) -ax-b$ admits at most one multiple zero. Theorem 2.6 finally enables us to conclude that the function $f(x)-ax=h(x)$ is prime.
	
\begin{theorem}\quad
	Let $f$ be a $p$-adic  transcendental  entire function. Then the set $\displaystyle \{a\in \mathbb{C}_{p} ;\, f(x)(x-a) \,\,\,is \, not \,prime\} $
	is at most a countable set.
\end{theorem}
To prove Theorem 2.17 we need the following lemma.  
\begin{lemma}
	Let $ f\in \mathcal{A}(\mathbb{C}_{p})\backslash\mathbb{C}_{p}[X] $. Then there exists a countable set $E'\subset \mathbb{C}_{p}$ such that for every $a\in \mathbb{C}_{p}\backslash E'$ and \ every \ $b\in\mathbb{C}_{p}$, the \ \ function $(x-a)f(x)-b$ has at most one multiple zero.
\end{lemma}
\begin{proof} 	\quad The procedure  is very similar   to that  used  in the \  proof \ of  Lemma 2.16. We easily check that an  element \ \  $\zeta$  of  $\mathbb{C}_{p}$ \ \  is a multiple  zero of the  function  $(x-a)f(x)-b$ \ if and only if
	\ $\left\{\begin{array}{ll}
	a=g(\zeta)\\ 
	b=h(\zeta)\\
	\end{array}\right.$, \ \ where $g$ and $h$ are meromorphic functions, \ defined in $\mathbb{C}_{p}$, \ by:
	
\smallskip	
\noindent$\left\{\begin{array}{ll}
	g(x)=x+f(x)/f'(x)\\ 
	h(x)=(x-g(x))f(x)\\
	\end{array}\right.$ \qquad\qquad\qquad\qquad\qquad\qquad\qquad\qquad\qquad\qquad (1)


	\medskip
Let \ $S(g')$ \ be the set of zeros and poles of  $g'$  \ and let $(d_i)_{i\ge1}$ be a countable family of disks  \ such that \ $\mathbb{C}_{p}\setminus S(g')=\cup_{i\ge1}d_i$ \ and that, \ for every $i\ge1$, \ \ the restriction \  $g_i$ \ of \  $g$ to $d_i$ is a bi-analytic function on $d_i$.  \  
 Then we have:  
	\ $\mathbb{C}_{p}=(\cup_{i\ge1} D_i) \cup g(S(g')), \ \text{where} \ D_i=g(d_i).$
	
	\medskip
	\noindent 
	As noted before, the disks $D_i$ are not necessarily pairwise disjoint. In other words,  some of them could be nested. To take account of this fact, let us set: 
	
	\centerline{$\Gamma=\{(i,j)\in (\mathbb{N}^{\ast})^2  /  D_{i}\subsetneqq D_{j} \ \text{and} \ h\circ g_i^{-1}\neq  h\circ g_j^{-1} \ \text{in} \  D_{i}\}$;}  
$\Delta_{ij}=\{x\in D_i  /   h\circ g_i^{-1}(x)=h\circ g_j^{-1}(x) \}, \forall (i,j)\in\Gamma$ \ and $\Delta=\cup_{(i,j)\in\Gamma}\Delta_{ij}$.\\
	Then \ $E'=\mathbb{C}_{p}\setminus [(\cup_{i\ge1}D_i)\setminus[(g(S(g')))\cup\Delta\cup(\cup_{i\ge1}Z(f\circ g_i^{-1}))]]$ \ is a \ countable subset of $\mathbb{C}_{p}$. 
\ Indeed, it is clear that \ $E'\subset(g(S(g')))\cup\Delta\cup(\cup_{i\ge1}Z(f\circ g_i^{-1}))]$.\\

		Let $a\in\mathbb{C}_{p} \backslash E'$. Let $I_a=\{i\in\mathbb{N}^{\ast} \ / \ a\in D_i \}$. \ Hence the disks $D_i$, for $i\in I_a\ $, are nested. We easily show, by relation (1), that the set of multiple zeros of the function $(x-a)f(x)-b$ \ is equal to $A=\{g_i^{-1}(a) \ / \ b=h\circ g_i^{-1}(a) \ \text{for} \ i\in I_a \}$.\\
	Suppose that $g_i^{-1}(a)$ and $g_j^{-1}(a)$ are two distinct elements of $A$. Hence, we have:  $h(g_i^{-1}(a))=h(g_j^{-1}(a))=b$. \ Assuming that $D_{i}\subsetneqq D_{j}$ and using the fact that $a\notin\Delta$, we have:
	
\noindent \ {$h\circ g_i^{-1}(x)=h\circ g_j^{-1}(x), \forall x\in D_i$} \qquad\qquad\qquad\qquad\qquad\qquad\qquad\qquad \ (2)

From relation (1), we have:

\noindent \ {$h'=-fg'$} \qquad\qquad\qquad\qquad\qquad\qquad\qquad\qquad \qquad\qquad\qquad\qquad\qquad (3) 
	 
	\noindent By derivation of relation (2) and using relation (3), we obtain:  
	
\noindent \ $f\circ g_i^{-1}(x)=f\circ g_j^{-1}(x),  \forall x\in D_i$ \qquad\qquad\qquad\qquad\qquad \qquad\qquad\qquad\quad \ (4) 	
	
\noindent Particularly, we have \ $f\circ g_i^{-1}(a)=f\circ g_j^{-1}(a).$ \ \
But, since  $a\in E'$, we have  \ $f\circ g_i^{-1}(a)=f\circ g_j^{-1}(a)\neq0$. \ Using this,\ we deduce from relation \ (1) \ that:\ $g_i^{-1}(a)=g_j^{-1}(a)$,  a contradiction.\ \ Hence the set $A$ admits at most one element. This completes the proof Lemma 2.18.
	\end{proof}
\noindent{\it Proof of Theorem 2.17 }

Let $E$ be the countable set of Lemme 2.18 and let $a\in \mathbb{C}_{p}\backslash E'$. \ Since $f$ is a  transcendental  entire function, we see that the function $h (x)=f(x)(x-a)$ is also a  transcendental  entire function. Lemma 2.18 then ensures that, for every $b\in\mathbb{C}_{p}$, \  the function $h (x) -b=f(x)(x-a)-b$ admits at most one multiple zero. Theorem 2.6 finally enables us to conclude that the function $f(x)(x-a)=h(x)$ is prime.

\section{PERMUTABILITY OF ENTIRE FUNCTIONS}

In this part, we consider the question of permutability of $p$-adic meromorphic  functions. In other words, given two $p$-adic meromorphic functions $f$ and $g$, under what conditions can we have $f \circ g = g\circ f$? The very particular situation that we are going to study gives an idea of the extreme difficulty of exploring this problem in a general way. \ More precisely, we will show the following result:

\begin{theorem}\quad Let $P$ and $f$ be respectively a non-constant polynomial and a transcendental entire function on $\mathbb{C}_{p}$ such that $P\circ f=f\circ P$. Then the polynomial $P$ is of one of the following forms:

\noindent {\it \ i}) \ $P(x)=x$,

\noindent {\it ii}) \ $P(x)=ax+b$, where $a, b\in\mathbb{C}_{p}$ such that $a$ is an $n$-th root of unity for some non-zero integer $n$. 
\end{theorem}
\medskip

To prove this theorem, we neen the following lemmas:

\medskip

\begin{lemma}
Let $f\in \mathcal{A}(\mathbb{C}_{p})$. Then the two following assertions are equivalent:

\noindent {\it \ i}) \ $f$ is a polynomial,

\noindent {\it ii}) \ there exist $c>0$ and  $d\geq0$  such that  $|f|(r)\leq cr^{d}$,    $r\to+\infty$. 

\end{lemma}
\begin{proof}
Suppose that $f$ is a polynomial of degree $n$: $ f (x) = a_{0}+ ...+ a_{n}x^{n}, a_{n} \not=0$. \\
Since, for $r$ large enough, we have $|f|(r)=|a_{n}|r^{n}$, just take
$c=|a_{n}|$ and $d=n$ to have the desired inequality. \\
Conversely, if there are $c>0$ and $d\geq0$ such that $|f|(r)\leq cr^{d}$ as $r\to+\infty$, we have for every integer $n> d $: \ 
$|f^{(n)}|(r)\leq{|f|(r)}/{r^{n}}\leq cr^{d-n} \longrightarrow0.$ \ 
So $f^{(n)}\equiv0$ and $f$ is a polynomial.
\end{proof}

\begin{lemma}
Let $f\in \mathcal{A}(\mathbb{C}_{p})$ and suppose that there exist real numbers $\alpha,\beta>0$ and an integer $n\geq2$ such that, 
$|f|(\alpha r^{n})<\beta (|f|(r))^{n},$ when $r\to+\infty$. \
Then $f$ is a polynomial. 
\end{lemma}
\begin{proof}
We will construct real numbers $c, d>0$ such that  $|f|(r)\leq cr^{d}$,   $r\to+\infty$.
Let $r_0$ be a real number such that  $r_{0}>\max \{1,\alpha^{{1}/({1-n})}\}.$ 
\ Let us choose  $c, \ d>0$ such that   
$\left\{
\begin{array}{ll}
|f|(r_{0})<cr_{0}^{d}\\ 
\beta<{\alpha^{d}}/{c^{n-1}}\\
\end{array} \right.$,  \  \
and let us  set $r_{k+1}=\alpha r_{k}^{n}$ for $k\geq0$.\\ The sequence $(r_{k})$ est strictly increasing and  tends to $+\infty$. \ We prove by induction that:
$|f|(r_{k})<cr_{k}^{d}, \,\,\forall k \geq0$. \
Indeed, the property holds  for $k=0$. \
Suppose that  it holds for some integer $k\ge0$.
Then we have:
 
\noindent$|f|(r_{k+1})=|f|(\alpha r_{k}^{n})<\beta(|f|(r_{k}))^{n}<\beta c^{n}r_{k}^{dn}=c(\beta c^{n-1})r_{k}^{dn}$

\qquad\quad $<c\alpha^{d}r_{k}^{dn}=c(\alpha r_{k}^{n})^{d}=c r_{k+1}^{d}$,

\noindent which means that this inequality is true for $k+1$.

\medskip

For every $r\in [r_{k},r_{k+1}]$, there exists $t\in[0; \ 1]$ such that  $r=r_{k}^{t} r_{k+1}^{1-t}$. \  As   $\log r \mapsto \log|f|(r)$ is a convex function, we deduce that:\

\medskip

$\begin{array}{cl}
|f|(r)&\leq(|f|(r_{k}))^{t}(|f|(r_{k+1}))^{1-t}=(|f|(r_{k}))^{t}(|f|(\alpha r_{k}^{n}))^{1-t} \\\\
&< (|f|(r_{k}))^{t}(\beta (|f|(r_{k}))^{n})^{1-t}<(c r_{k}^{d})^{t}(\beta(cr_{k}^{d})^{n})^{1-t}\\\\
&= c^{t}r_{k}^{dt}c^{1-t}(\beta c^{n-1})^{1-t}(r_{k}^{n(1-t)})^{d} < c r_{k}^{dt} \alpha^{d(1-t)}  r_{k}^{n(1-t)d}\\\\
& =c(r_{k}^{t}(\alpha r_{k}^{n})^{1-t})^{d}=c(r_{k}^{t}r_{k+1}^{1-t})^{d}=c r^{d}.
\end{array}$\\

Hence $|f|(r)<cr^{d}, \,\, \forall r\geq r_{0}$. \
By Lemma $3.2$, $f$ is then a polynomial. 
\end{proof}
\medskip
\noindent{\it Proof of Theorem 3.1.}

Let us set $P(x)= a_0+\dots+a_kx^k$, where $a_0, \dots, a_k$ are elements of $\mathbb{C}_{p}$ and $a_k\neq 0$. \ Then: $(f\circ P)(x)=(P\circ f)(x)= a_0+\dots+a_k(f(x))^k$. \ Suppose that $k\geq2$. Then, for  $r>0$, we have $|f|(|P|(r))=|P|(|f|(r)).$ 
\ It follows that,  for  sufficiently large $r>0$, we have $|f|(|a_k|r^k)=|a_k|(|f|(r))^k<2|a_k|(|f|(r))^k$.
\ Lemma 3.3 then implies that $f$ is a polynomial, , which is a contradiction.   Hence, we have $k=1$ and $P$ is of the form $P(x)=ax+b$, where $a, b\in\mathbb{C}_{p}$ and $a\neq0$.
\ Two cases can then arise:

\noindent {\it \ i})\quad $a=1$.  Then we have $b=0$. Indeed suppose that $b\neq0$. Then we have, for every $x\in\mathbb{C}_{p}$, $f(x+b)=f(x)+b$. It follows that $f'(x+b)=f'(x)$, for every $x\in\mathbb{C}_{p}$. Let $\zeta$ be a zero of $f'$ such that $|\zeta|>|b|$. Then $\zeta,  \zeta+b, \zeta+2b, \dots$ are infinitely many zeros of $f'$ included in the disk $d(0; |\zeta|)$, a contradiction. Hence in this case we have $P(x)=x$.

\noindent {\it ii})\quad $a\neq1$.  \   
Let $\sigma$ be the affine application $\sigma(t)=t+{b}({a-1})^{-1}$, hence its inverse$\sigma^{-1}$  is given by  $\sigma^{-1} (t)=t+{b}({1-a})^{-1}$. \ Let $F$ be  the function given by $F=\sigma\circ f\circ\sigma^{-1}$. \ It is easily seen that: \ $F(ax)= aF(x), \ \forall x\in\mathbb{C}_{p}$.
 
\noindent If $F(x)=\sum_{n\geq 0}b_{n}x^{n} $, we have $F(ax)=\sum_{n\geq 0}b_{n}a^{n}x^{n} $. It follows that:

 $\sum_{n\geq 0}(a^{n}-a)b_{n}x^{n}=0$, and hence $b_{0}=0$ and $(a^{n}-a)b_{n+1}=0,\,\, \forall n\geq1$.
 
 \medskip 
 
\noindent Suppose that
there exist two relatively prime integers $m, \ n\geq2$ such that $b_{n+1}\not=0$, and $b_{m+1}\not=0$
we would have $a^{n}=a^{m}=1$, and therefore $a=1$, which excluded.  Hence $F$ has the form $F(X)=
\sum_{k\geq 1}b_{nk+1}x^{nk+1}$, \ where $n$ is the smallest positive integer such that  $a^{n}=1$. \ It follows that:
  
{$f(x)=\sum_{k\geq 1}b_{nk+1}(x+{b}/({a-1}))^{nk+1}+{b}/({1-a}).$} \qquad\qquad\qquad\qquad\qquad \sqw
 
\medskip

\begin{corollary}
Let $f \in \mathcal{A}(\mathbb{C}_{p}) \backslash\mathbb{C}_{p}[X]$, $f(x)=\sum_{n\geq 0}a_{n}x^{n} $. Suppose that there exist two relatively prime integers $m,l$ such that  $a_{m}a_{l}\not=0$. Then the only non-constant polynomial $P$ such that $P\circ f = f\circ P $ is $ P (x) = x $.	
	
\end{corollary}


Bilal SAOUDI 

Laboratoire de Math\'ematiques Pures et Appliqu\'ees,

  Universit\'e de Jijel ,
  
   BP 98, Jijel, Alg\'erie. 
  
 Email:saoudibilal1@gmail.com \\

Abdelbaki BOUTABAA

Universit\'e Clermont Auvergne LMBP UMR 6620 - CNRS

Campus universitaire des Cézeaux
 
3, Place Vasarely

TSA 60026

CS 60026

63178, Aubière-Cedex; France

Email: Abdelbaki.Boutabaa@math.univ-bpclermont.fr\\

Tahar Zerzaihi

Laboratoire de Math\'ematiques Pures et Appliqu\'ees,

Universit\'e de Jijel ,

 BP 98, Jijel, Alg\'erie.

Email:zerzaihi@yahoo.com \\

\label{lastpage}

\end{document}